\documentclass[12pt,reqno]{amsart}
\pdfoutput=1
\usepackage[margin=1in]{geometry}
\usepackage{amsmath,amssymb,amsthm,amsfonts}
\usepackage{mathtools}
\usepackage{thmtools}
\usepackage{listings}
\usepackage{mathrsfs}
\usepackage{tikz-cd}
\usepackage{mathdots}
\usepackage{accents}
\usepackage{yhmath}
\usepackage{enumerate}
\usepackage{cancel}
\usepackage{diagbox}
\usepackage{shuffle}
\usepackage{quiver}
\usepackage{bbm}

\renewcommand{\empty}{\varnothing}

\newcommand{\eps}{\varepsilon}

\newcommand{\comp}{\vDash}
\newcommand{\til}{\widetilde}
\newcommand{\rev}[1]{\reflectbox{\ensuremath{\vec{\reflectbox{\ensuremath{#1}}}}}} 
\newcommand{\f}{\frac}
\newcommand{\A}{\[\begin{aligned}}
\newcommand{\B}{\end{aligned}\]}
\newcommand{\Zpos}{\mathbb{Z}_{>0}}

\newcommand{\Comp}{\mathrm{Comp}}
\newcommand{\Compne}{\mathrm{Comp}_{>0}}
\newcommand{\Sym}{\mathrm{Sym}}
\newcommand{\NSym}{\mathrm{NSym}}
\newcommand{\QSym}{\mathrm{QSym}}
\newcommand{\Sh}{\mathrm{Sh}}

\newcommand{\sh}{\shuffle}
\renewcommand{\a}{\alpha}
\renewcommand{\b}{\beta}
\newcommand{\y}{\gamma}
\newcommand{\tensor}{\otimes}

\newcommand{\bk}{\mathbbm{k}}
\newcommand{\bX}{\mathbb{X}}

\DeclareMathOperator{\p}{prod}
\DeclareMathOperator{\aut}{aut}
\DeclareMathOperator{\even}{e}
\DeclareMathOperator{\odd}{o}
\DeclareMathOperator{\lp}{lp}

\newtheorem{thm}{Theorem}[section]
\newtheorem{prop}[thm]{Proposition}

\newtheorem{lemma}[thm]{Lemma}
\theoremstyle{definition}

\newtheorem{example}[thm]{Example}
\newtheorem{defn}[thm]{Definition}

\theoremstyle{remark}
\newtheorem{remark}[thm]{Remark}

\newcommand{\ph}{\varphi}

\renewcommand{\angle}[1]{\langle#1\rangle}

\usepackage{biblatex} 
\usepackage{url}

\numberwithin{equation}{section}

\title{Shuffle bases and quasisymmetric power sums}
\author{Ricky Ini Liu}
\address{Department of Mathematics, University of Washington, Seattle, WA 98195}
\email{riliu@uw.edu}

\author{Michael Tang}
\address{Department of Mathematics, University of Washington, Seattle, WA 98195}
\email{mst0@uw.edu}
\date{\today}

\begin{document}
\maketitle

\begin{abstract}
    The algebra of quasisymmetric functions $\QSym$ and the shuffle algebra of compositions $\Sh$ are isomorphic as graded Hopf algebras (in characteristic zero), and isomorphisms between them can be specified via \emph{shuffle bases} of $\QSym$. We use the notion of \emph{infinitesimal characters} to characterize shuffle bases, and we establish a universal property for $\Sh$ in the category of connected graded Hopf algebras equipped with an infinitesimal character, analogous to the universal property of $\QSym$ as a combinatorial Hopf algebra described by Aguiar, Bergeron, and Sottile.
        We then use these results to give general constructions for quasisymmetric power sums, recovering four previous constructions from the literature, and study their properties.
\end{abstract}

\section{Introduction}

The shuffle algebra $\Sh$ is a connected graded Hopf algebra with a graded basis $\{x_\a\}$ indexed by compositions $\a,$ where multiplication is given by the shuffle of compositions, $x_\a x_\b = \sum_{\y \in \a \sh \b} x_\y,$ and comultiplication is given by deconcatenation, $x_\y = \sum_{\a \b = \y} x_\a \tensor x_\b.$ In characteristic zero, it is known that $\Sh$ is isomorphic as a graded Hopf algebra to the more well-known algebra $\QSym$ of quasisymmetric functions (see \cite[Exercise 5.4.12]{grinberg} or \cite[Theorem 2.1]{malvenuto}). In this paper, we study the structure of $\Sh$ and, in particular, its relationship with $\QSym$.

The algebra $\QSym$ is ubiquitous in that many types of combinatorial objects (posets, graphs, matroids, etc.)\ are naturally associated with certain quasisymmetric invariants. Aguiar, Bergeron, and Sottile \cite{abs} explain this phenomenon by describing a universal property of $\QSym$. Specifically, they show that for a certain character $\zeta_Q \colon \QSym \to \bk$ (where $\bk$ is the base field), the pair $(\QSym, \zeta_Q)$ is the terminal object in the category of \emph{combinatorial Hopf algebras} (Hopf algebras equipped with distinguished characters). Quasisymmetric invariants then arise naturally by applying the unique Hopf morphism to $(\QSym, \zeta_Q)$ from a combinatorial Hopf algebra of interest.

To study the relationship between $\Sh$ and $\QSym$, we first establish an analogous universal property of $\Sh$. In Section~\ref{sec:infchars}, we define the category of \emph{infinitorial Hopf algebras}, which are Hopf algebras equipped with distinguished \emph{infinitesimal characters}. Then we show that, for a certain infinitesimal character $\xi_S$ of $\Sh$, the pair $(\Sh, \xi_S)$ is the terminal object in this category. Using both universal properties, we demonstrate how isomorphisms between $\Sh$ and $\QSym$ give rise to bijections between the group $\bX(H)$ of characters and the Lie algebra $\Xi(H)$ of infinitesimal characters of any connected graded Hopf algebra $H$.

Next, in Section~\ref{sec:shuffle-bases}, we turn to the study of the isomorphisms between $\Sh$ and $\QSym$ themselves. Under such an isomorphism, the basis $\{x_\a\}$ is sent to a graded basis $\{X_\a\}$ of $\QSym$ with the same multiplicative and comultiplicative structure; we call such a basis a \emph{shuffle basis}. Again using the universal properties of $\Sh$ and $\QSym$, we give characterizations of shuffle bases in terms of characters of $\Sh$ and infinitesimal characters of $\QSym$.

In Section~\ref{sec:shuffle-qps}, we demonstrate how shuffle bases are closely related (in fact, nearly equivalent) to another class of graded bases for $\QSym$ called \emph{quasisymmetric power sum} (QPS) bases. Quasisymmetric power sums are quasisymmetric analogs and refinements of the symmetric power sum basis $\{p_\lambda\},$ which arise naturally by duality from the noncommutative power sums (see \cite{gelfand}, as well as \cite{as,bdhmn,liuweselcouch} for some combinatorial applications). We use our characterizations of shuffle bases to give general constructions for QPS bases and, in doing so, recover several bases that have been studied previously. Namely, using a construction we call the ``prefix sum construction,'' we recover the so-called ``type I'' and ``type II'' quasisymmetric power sums, whose duals were introduced in \cite{gelfand} and which were further studied in \cite{bdhmn}. Then, using a construction we call the ``ordered partition construction,'' we recover two more QPS bases that are less well-known yet still exhibit interesting properties: the ``shuffle'' basis of \cite{al} (which we will instead call the ``even-odd'' basis), and the ``combinatorial'' basis defined in \cite{awvw} and \cite{lazzeroni}.

In Section~\ref{sec:qps-properties}, we give additional properties and characterizations of these QPS bases. We rederive the expansions of the monomial basis $\{M_\a\}$ of $\QSym$ in the type I and type II bases, and we show how the type II basis gives rise to the well-known \emph{exponential map} from $\Xi(H)$ to $\bX(H)$. We give a family of QPS bases that generalizes the aforementioned even-odd basis and which are all eigenbases for the \emph{theta map} $\Theta \colon \QSym \to \QSym.$ Finally, we determine all QPS bases which expand into the monomial basis with nonnegative integer coefficients.

\section{Preliminaries} \label{prelim}

Throughout this paper, we fix a field $\bk$, which is understood to be the base field for all vector spaces.

\subsection{Compositions} \label{compositions}

A \emph{composition} is a finite sequence $\a = (\a_1, \dots, \a_\ell)$ of positive integers; each $\a_i$ is called a \emph{part} of $\a.$
The \emph{size} of $\a$ is $|\a| = \a_1 + \dots + \a_\ell.$ If $|\a| = n,$ we say that $\a$ is a \emph{composition of $n$} and write $\a \comp n.$ The \emph{length} of a composition $\a,$ denoted $\ell(\a),$ is the number of its parts. For each positive integer $i \in \Zpos,$ we define $m_i(\a)$ to be the number of times $i$ appears as a part of $\a$.

We denote by $\empty$ the unique composition with no parts; thus $|\empty| = \ell(\empty) = 0.$ Additionally, we will often identify positive integers $n$ with the corresponding composition $(n)$ of length $1.$ We denote the set of all compositions by $\Comp,$ the set of nonempty compositions by $\Compne = \Comp \setminus \{\empty\},$ and the set of compositions of $n \ge 0$ by $\Comp_n.$

Given a composition $\a = (\a_1, \dots, \a_\ell),$ its \emph{reverse} is the composition $\rev{\a} = (\a_\ell, \dots, \a_1)$ of the same size and length as $\a$. We say a composition is \emph{even} (resp.\ odd) if all of its parts are even (resp.\ odd).

A \emph{partition} is a composition $\a = (\a_1, \dots, \a_\ell)$ whose parts are weakly decreasing, that is, $\a_1 \ge \ldots \ge \a_\ell.$ For compositions $\a$ and $\b,$ we write $\a \sim \b$ if $\a$ can be obtained by rearranging the parts of $\b$ (or equivalently, $m_i(\a) = m_i(\b)$ for all $i$). Thus, for each composition $\a,$ there is a unique partition $\lambda$ such that $\a \sim \lambda$; we denote this partition by $\tilde{\a}.$

The \emph{concatenation} of two compositions $\a = (\a_1, \dots, \a_k)$ and $\b = (\b_1, \dots, \b_\ell)$ is the composition $\a\b = (\a_1, \dots, \a_k, \b_1, \dots, \b_\ell).$ When $|\a| = |\b|,$ we say that $\a$ \emph{refines} $\b$ (and $\b$ \emph{coarsens} $\a$), denoted $\a \le \b$, if $\a$ can be written as a concatenation
\begin{equation} \label{eq:refinement}
    \a = \a^{(1)} \dotsm \a^{(\ell(\b))} \qquad \text{where } \a^{(i)} \comp \b_i \text{ for each } i.
\end{equation} (When $\a \le \b,$ we will typically use the notation $\a^{(i)}$ as above without further comment.) This defines a partial order $\le$ on $\Comp_n$ for each $n.$

\subsubsection{Functions on compositions} \label{sec:compfunctions}

We define the following functions on compositions.

\begin{defn}
    For a composition $\a = (\a_1, \dots, \a_\ell)$, define:
    \begin{enumerate}[(i)]
        \item $\lp(\a) = \a_\ell,$ the last part of $\a$;
        \item $\p(\a) = \a_1 \dotsm \a_\ell,$ the product of the parts of $\a$;
        \item $\aut(\a) = \prod_{i \ge 1} m_i(\a)!$; and
        \item $z_\a = \p(\a) \aut(\a).$
    \end{enumerate}
    
    (By convention, $\lp(\empty) = 0$ and $\p(\empty) = \aut(\empty) = z_\empty = 1.$)
\end{defn}

Thus, $\aut(\a)$ is the number of ways to permute equal parts of $\a.$ It follows from the orbit-stabilizer theorem that the number of compositions $\b$ with $\a \sim \b$ is $\ell(\a)! / \aut(\a),$ and that the number of permutations on $|\a|$ elements with cycle type $\tilde{\alpha}$ is $|\a|! / z_\a.$

Given a function on compositions, we can extend it to a function on a pair of compositions $\alpha \le \beta$ in the following way.
\begin{defn} \label{def:f-alphabeta}
    Let $f \colon \Compne \to \bk$ be a function. For compositions $\a$ and $\b$ with $\a \le \b,$ we define \A
        f(\a, \b) = f(\a^{(1)}) \dotsm f(\a^{(\ell(\b))}),
    \B where $\a^{(1)}, \dots, \a^{(\ell(\b))}$ are given by \eqref{eq:refinement}. (By convention, $f(\empty, \empty) = 1.$)
\end{defn}

Note that this respects concatenation: $f(\a\a', \b\b') = f(\a,\b) f(\a',\b')$ when $\a \le \b$ and $\a' \le \b'.$ Also, for $\a \neq \empty,$ we have $f(\a, |\a|) = f(\a).$

\subsection{Hopf algebras and duality}

Let $H = (H, m_H, u_H, \Delta_H, \eps_H, S_H)$ be a Hopf algebra over $\bk.$ The identity element will be denoted $1_H = u_H(1)$. (We will omit subscripts when there is no cause for confusion.) For an integer $k \ge 0,$ we write $\Delta^{k-1} \colon H \to H^{\tensor k}$ for the map obtained by iterating comultiplication $k-1$ times (where by convention $\Delta^{-1} = \eps$).

We will consider only connected graded Hopf algebras of finite type, so that all structure maps are graded with respect to the direct sum decomposition $H = \bigoplus_{n \ge 0} H_n$, $\dim H_n < \infty$ for all $n$, and $H_0 \cong \bk$.

Given a composition $\a = (\a_1, \dots, \a_\ell) \comp n$ and an element $h \in H$, we define $\Delta_\a(h)$ to be the component of $\Delta^{\ell-1}(h) \in H^{\tensor \ell}$ in degree $(\a_1, \dots, \a_\ell),$ that is, the projection of $\Delta^{\ell-1}(h)$ onto $H_{\a_1} \tensor \dotsm \tensor H_{\a_\ell}.$

If $V$ is a vector space, we denote by $V^*$ its dual vector space. If $V = \bigoplus_{n \ge 0} V_n$ is a graded vector space, then its \emph{graded dual} is $V^\circ = \bigoplus_{n \ge 0} V_n^*.$ If $\{v_\a\}_{\a \in I}$ is a graded basis for $V$ (and $V$ is of finite type), then we denote by $\{v_\a^*\}_{\a \in I}$ the dual basis of $V^\circ$, so that $\angle{v_\a^*, v_\b} = \delta_{\a\b}$ for all $\a, \b \in I.$ 

\subsection{Quasisymmetric functions}

A \emph{quasisymmetric function} (over $\bk$) is a formal power series $F \in \bk[[x_1, x_2, \dots]]$ of bounded degree in variables $x_1, x_2, \dots$ such that the coefficients of $x_{i_1}^{\a_1} \dotsm x_{i_k}^{\a_k}$ and $x_{j_1}^{\a_1} \dotsm x_{j_k}^{\a_k}$ in $F$ are equal whenever $i_1 < \dots < i_k$ and $j_1 < \dots < j_k.$

The quasisymmetric functions naturally form a vector space (over $\bk$) denoted $\QSym.$ A standard basis of $\QSym$ is the \emph{monomial basis} $\{M_\a\},$ which is indexed by compositions $\a = (\a_1, \dots, \a_\ell)$ and defined by \[
    M_\a = \sum_{i_1 < \dots < i_\ell} x_{i_1}^{\a_1} \dotsm x_{i_\ell}^{\a_\ell}
\] (and $M_\empty = 1$). In fact, $\QSym$ has the structure of a connected graded Hopf algebra, where the grading is given by $\QSym_n = \operatorname{span} \{ M_\a : \a \comp n\},$ multiplication is the usual multiplication of power series, and comultiplication is given by deconcatenation:
\begin{equation} \label{eq:monomial-deconcatenate}
    \Delta(M_\y) = \sum_{\a \b = \y} M_\a \tensor M_\b = \sum_{i=0}^{\ell(\y)} M_{(\y_1, \dots, \y_i)} \tensor M_{(\y_{i+1}, \dots, \y_{\ell(\y)})}.
\end{equation} 
The ring $\Sym$ of symmetric functions is naturally contained in $\QSym$ 
(since, for example, the \emph{(symmetric) power sum basis} $\{p_\lambda\}$ satisfies $p_\lambda = \prod_{i=1}^{\ell(\lambda)} M_{\lambda_i}$),
which also makes $\Sym$ a connected graded Hopf algebra.

\subsection{Combinatorial Hopf algebras} 

Here we summarize some important results from the work of Aguiar, Bergeron, and Sottile \cite{abs}.

A \emph{character} of a Hopf algebra $H$ is an algebra map $\zeta \colon H \to \bk,$ that is, a linear map satisfying $\zeta(ab) = \zeta(a) \zeta(b)$ and $\zeta(1_H) = 1.$ For any Hopf algebra $H,$ its characters form a group $\bX(H)$ under \emph{convolution}: \[
    \zeta * \zeta' = m_\bk \circ (\zeta \tensor \zeta') \circ \Delta_H.
\] The identity element of $\bX(H)$ is the counit $\eps_H,$ and the inverse of $\zeta \in \bX(H)$ is $\zeta^{-1} = \zeta \circ S_H.$

\begin{defn}
    A \emph{combinatorial Hopf algebra} is a pair $(H, \zeta),$ where $H$ is a connected graded Hopf algebra and $\zeta \in \bX(H).$ A \emph{morphism of combinatorial Hopf algebras} $(H, \zeta) \to (H', \zeta')$ is a graded Hopf map $\Phi \colon H \to H'$ that also satisfies $\zeta = \zeta'\circ \Phi.$
\end{defn}

Let $\zeta_Q$ be the character of $\QSym$ defined on the monomial basis by \[
    \zeta_Q(M_\a) = \begin{cases}
        1 & \text{if } \ell(\a) \le 1, \\
        0 & \text{otherwise},
    \end{cases}
\] or equivalently on the level of formal power series by \[
    \zeta_Q(F) = F(1, 0, 0, \dots).
\]
By \cite[Theorem 4.1]{abs}, the pair $(\QSym, \zeta_Q)$ is the terminal object in the category of combinatorial Hopf algebras; in fact, it is also the terminal object in the category of \emph{combinatorial coalgebras} (connected graded coalgebras equipped with linear functionals). For our purposes, we restate this result, in a slightly less general form, as follows:

\begin{prop} \label{prop:abs-universal}
    Let $(H, \zeta)$ be any combinatorial Hopf algebra.
    
    \begin{enumerate}[(i)]
        \item There is a unique graded coalgebra map $\Phi \colon H \to \QSym$ satisfying $\zeta = \zeta_Q \circ \Phi$.
        \item The map $\Phi$ is also an algebra map; hence, it is a map of combinatorial Hopf algebras $(H, \zeta) \to (\QSym, \zeta_Q)$ and is the unique such map.
        \item The map $\Phi$ is given as follows: for $h \in H_n,$ we have \[
            \Phi(h) = \sum_{\a \comp n} \left( \zeta^{\tensor \ell(\a)} \Delta_\a(h) \right) M_\a.
        \]
    \end{enumerate}
\end{prop}

This result explains why many classes of combinatorial objects have natural quasisymmetric invariants. We now present one important example of this phenomenon. (Many more examples appear in \cite{abs}.)

\begin{example} \label{ex:posets0}
     Let $\mathcal{P}$ be the Hopf algebra of finite posets (up to isomorphism), graded by size, whose multiplication is disjoint union and whose comultiplication is given by \[
        \Delta(P) = \sum_{I} I \tensor (P \setminus I)
    \] where $I$ ranges over all order ideals of $P$. Then $\mathcal{P}$ can be made into a combinatorial Hopf algebra by taking the character $\zeta$ given by $\zeta(P) = 1$ for all posets $P.$ Using Proposition~\ref{prop:abs-universal}(iii), one can check that the unique map of combinatorial Hopf algebras $\Phi\colon (\mathcal{P}, \zeta) \to (\QSym, \zeta_Q)$ sends a poset $P$ to its (naturally labeled) \emph{$P$-partition generating function}: \[
        \Phi(P) = K_P(\mathbf{x}).
    \] (See e.g.\ \cite{mcnamaraward} for more information about $P$-partition generating functions.)
\end{example}

\subsection{The shuffle algebra} \label{section-shufflealgebra}

We now recall the notion of the shuffle of two compositions and use it to define the shuffle algebra, the other central object of our study.

\begin{defn}
    Let $\a = (\a_1, \dots, \a_k)$ and $\b = (\b_1, \dots, \b_\ell)$ be compositions. Their \emph{shuffle set}, denoted $\a \sh \b,$ is the multiset consisting of all compositions of length $k+\ell$ that contain $\a$ and $\b$ as disjoint subsequences, taken with multiplicity.
\end{defn}

\begin{defn}
    The \emph{shuffle algebra}\footnote{
        Other sources instead define ``the shuffle algebra $\Sh(V)$ of a vector space $V$'' (e.g. \cite[Proposition 1.6.7]{grinberg}). From this point of view, the shuffle algebra we are concerned with here is the shuffle algebra of a vector space $V$ with a basis indexed by the positive integers.
    } $\Sh$ is the connected graded Hopf algebra with basis $\{x_\a\}$ indexed by compositions $\a,$ where the grading is given by $\Sh_n = \operatorname{span} \{ x_\a : \a \comp n\},$ multiplication is given by shuffling: \begin{equation}
        x_\a x_\b = \sum_{\y \in \a \sh \b} x_\y,
    \end{equation} and comultiplication is given by deconcatenation: \begin{equation} \label{eq:shuffle-deconcatenate}
        \Delta(x_\y) = \sum_{\a \b = \y} x_\a \tensor x_\b.
    \end{equation}
\end{defn}

For example, we have \[
    x_{12} x_{2} = 2x_{122} + x_{212} \qquad \text{ and } \qquad \Delta(x_{31}) = 1 \tensor x_{31} + x_{3} \tensor x_{1} + x_{31} \tensor 1
\] (where we omit parentheses and commas when writing compositions; also note that $x_\empty = 1$ is the identity element of $\Sh$). By \cite[Theorem 3.1]{benedettisagan}, the antipode of $\Sh$ is \begin{equation} \label{eq:Sh-antipode}
    S_\Sh(x_\a) = (-1)^{\ell(\a)} x_{\tiny \rev{\a}}.
\end{equation}

\section{Infinitesimal characters of Hopf algebras} \label{sec:infchars}

In this section, we prove a universal property of the shuffle algebra $\Sh$ analogous to that of $\QSym$ proved by Aguiar, Bergeron, and Sottile \cite{abs}. We first recall the definition of an infinitesimal character of a Hopf algebra and give some examples; see \cite{manchon2006hopf} for more information.

\begin{defn}
    Let $H$ be a Hopf algebra. An \emph{infinitesimal character} of $H$ is a linear map $\xi \colon H \to \bk$ that satisfies \[
        \xi(ab) = \eps(a) \xi(b) + \xi(a) \eps(b)
    \] for all $a, b \in H.$ 
    We denote the set of all infinitesimal characters of $H$ by $\Xi(H).$
\end{defn}

When $H$ is graded and connected, we can give a simpler equivalent definition.

\begin{prop} \label{prop:infchar}
    Let $H$ be a connected graded Hopf algebra. A linear map $\xi \colon H \to \bk$ is an infinitesimal character of $H$ if and only if $\xi(1_H) = 0$ and $\xi(ab) = 0$ for all homogeneous $a, b \in H$ of positive degree.
\end{prop}

\begin{proof}
    If $\xi$ is an infinitesimal character, then $\xi(1_H) = \xi(1_H 1_H) = 2 \eps(1_H) \xi(1_H) = 2 \xi(1_H),$ so $\xi(1_H) = 0.$ Since $H$ is graded, if $a,b \in H$ are homogeneous of positive degree, $\eps(a)=\eps(b)=0$, so $\xi(ab) = \eps(a) \xi(b) + \xi(a) \eps(b) = 0$.

    For the converse, by bilinearity, it suffices to show that $\xi(ab) = \eps(a) \xi(b) + \xi(a) \eps(b)$ for homogeneous $a$ and $b$. If $a$ and $b$ both have positive degree, then both sides vanish. Otherwise, since $H$ is connected, we may assume $a = 1_H$ by rescaling. Then \[
        \eps(1_H) \xi(b) + \xi(1_H) \eps(b) = 1 \cdot \xi(b) + 0 \cdot \eps(b) = \xi(b) = \xi(1_H \cdot b),
    \] as desired.
\end{proof}

\begin{remark}In the connected graded case, $\Xi(H)$ has the structure of a Lie algebra (over $\bk$) with Lie bracket \[
    [\xi, \xi'] := \xi * \xi'- \xi' * \xi,
\] where $*$ is convolution (\cite[Proposition II.4.2]{manchon2006hopf}). In fact, $\bX(H)$ can be made into an infinite-dimensional Lie group with $\Xi(H)$ as the corresponding Lie algebra; we refer the interested reader to \cite{infdimliegroups} for the construction.
\end{remark}

Here are some examples of infinitesimal characters of combinatorial interest.

\begin{example} \label{example-posets1}
    Let $\mathcal{P}$ be the Hopf algebra of posets defined in Example~\ref{ex:posets0}. Consider the linear map $\xi \colon \mathcal{P} \to \bk$ defined on posets $P$ by \[
        \xi(P) = \begin{cases}
            1 & \text{if } P \text{ has a unique minimal element}, \\
            0 & \text{otherwise}.
        \end{cases}
    \] Then $\xi$ is an infinitesimal character of $\mathcal{P}$ because it annihilates the empty poset, as well as the disjoint union of two nonempty posets since this never has a unique minimal element.
\end{example}

\begin{example} \label{example-graphs1}
    Let $\mathcal{G}$ be the Hopf algebra of finite graphs (up to isomorphism), graded by number of vertices, whose multiplication is disjoint union and whose comultiplication is given by \[
        \Delta(G) = \sum_{H \subseteq G} H \tensor (G \setminus H),
    \] where $H$ and $G \setminus H$ are induced subgraphs on complementary subsets of vertices. 
    Consider the linear map $\xi\colon \mathcal{G} \to \bk$ that sends a graph $G$ to the coefficient of $k^1$ in its chromatic polynomial $\chi_G(k).$ Then $\xi$ is an infinitesimal character of $\mathcal{G}.$ Indeed, $\xi$ annihilates the empty graph (whose chromatic polynomial is $1$); and $k$ divides $\chi_G(k)$ for every nonempty graph $G,$ so $k^2$ divides $\chi_G(k) \chi_H(k) = \chi_{G \sqcup H}(k)$ for any nonempty graphs $G$ and $H,$ which implies $\xi(G \sqcup H) = 0.$
\end{example}

One might wonder why we are choosing to highlight these examples---after all, a more obvious example of an infinitesimal character of $\mathcal{P}$ (resp.\ $\mathcal{G}$) is the linear map that simply sends nonempty connected posets (resp.\ graphs) to $1$ and all others to $0.$ One reason is that these examples have other significance: Example~\ref{example-posets1} was used in \cite{liuweselcouch} to prove the irreducibility of $P$-partition generating functions of connected posets, and Example~\ref{example-graphs1} was shown in \cite{greenezaslavsky} to count certain acyclic orientations of graphs. Another reason, which we will explore below, is that these infinitesimal characters arise from very simple characters of $\mathcal{P}$ and $\mathcal{G}$ under certain bijections $\bX(\mathcal{P}) \to \Xi(\mathcal{P})$ and $\bX(\mathcal{G}) \to \Xi(\mathcal{G}).$

Our final example of an infinitesimal character is one for the shuffle algebra $\Sh,$ which will play a similar role as the canonical character $\zeta_Q$ of $\QSym.$

\begin{example}
    Define $\xi_S \colon \Sh \to \bk$ by \[
        \xi_S(x_\a) = \begin{cases}
            1 & \text{if } \ell(\a) = 1, \\
            0 & \text{otherwise.}
        \end{cases}
    \]
    Then $\xi_S$ is an infinitesimal character. Indeed, $\xi_S(x_\empty) = 0$ since $\ell(\empty) = 0$; and if $\a$ and $\b$ are nonempty compositions, then \[
        \xi_S(x_\a x_\b) = \sum_{\y \in \a \sh \b} \xi_S(x_\y) = 0
    \] because $\ell(\y) = \ell(\a) + \ell(\b) \ge 2$ for all $\y \in \a \sh \b.$
\end{example}

\subsection{Infinitorial Hopf algebras}

We now construct the category of infinitorial Hopf algebras and show that the shuffle algebra $\Sh$ is the terminal object in this category. This provides a natural analog of Proposition~\ref{prop:abs-universal} and places $\Sh$ on equal footing with $\QSym.$ 
\begin{defn}
    An \emph{infinitorial Hopf algebra} is a pair $(H, \xi),$ where $H$ is a connected graded Hopf algebra and $\xi \in \Xi(H).$ A \emph{morphism of infinitorial Hopf algebras} $(H, \xi) \to (H', \xi')$ is a graded Hopf map $\Phi \colon H \to H'$ that also satisfies $\xi = \xi' \circ \Phi.$
\end{defn}

\begin{example} \label{ex:infinitorial-morphism}
    Let $\mathcal{P}$ and $\xi$ be as in Example~\ref{ex:posets0}, and define a linear map $\eta \colon \QSym \to \bk$ by $\eta(M_\a) = (-1)^{\ell(\a)-1} \lp(\a).$ Liu and Weselcouch \cite[\S 3.1]{liuweselcouch} showed that $\eta$ is an infinitesimal character of $\QSym$ and that \[
        \eta(K_P(\mathbf{x})) = \xi(P),
    \] where $K_P(\mathbf{x})$ is the (naturally labeled) $P$-partition generating function.
    
    Recall from Example~\ref{ex:posets0} that $\Phi(P) = K_P(\mathbf x)$ is a graded Hopf map $\mathcal{P} \to \QSym$; since $\eta \circ \Phi = \xi,$ it follows that $\Phi$ is a map of infinitorial Hopf algebras $(\mathcal{P}, \xi) \to (\QSym, \eta).$ (We will revisit this example in Section~\ref{sec:typeI}.)
\end{example}

With these definitions, we obtain the category of infinitorial Hopf algebras (over $\bk$). To prove that $(\Sh, \xi_S)$ is the terminal object in this category, we will first need a simple lemma.

\begin{lemma} \label{lem:infinitorial-tensor}
    If $H$ is connected graded and $\xi \in \Xi(H),$ then $\xi \circ m_H \in \Xi(H \tensor H).$
\end{lemma}

\begin{proof}
        First, we have $\xi \circ m_H(1_{H \tensor H}) = \xi \circ m_H(1_H \tensor 1_H) = \xi(1_H) = 0$, as needed. Now let $a \tensor b$ and $a' \tensor b'$ be homogeneous positive-degree elements of $H \tensor H.$ Then \[
        (\xi \circ m_H)\left((a\tensor b)(a'\tensor b')\right)
        = (\xi \circ m_H)\left(aa' \tensor bb'\right)
        = \xi(aa'bb').
    \] Because $\deg (a \tensor b) > 0,$ either $\deg a > 0$ or $\deg b > 0$; the same holds for $a'$ and $b'.$ Therefore, $aa'bb'$ can always be written as the product of two homogeneous positive-degree elements (either $a \cdot a'bb',$ $aa' \cdot bb',$ or $aa'b \cdot b'$). Therefore, $\xi(aa'bb') = 0,$ as desired. This completes the proof by Proposition~\ref{prop:infchar}.
\end{proof}

We are now ready to prove the main result of this section.

\begin{thm} \label{thm:Sh-universal}
    Let $(H, \xi)$ be any infinitorial Hopf algebra.
    
    \begin{enumerate}[(i)]
        \item There is a unique graded coalgebra map $\Psi \colon H \to \Sh$ satisfying $\xi = \xi_S \circ \Psi$.
        \item The map $\Psi$ is also an algebra map; hence, it is a map of infinitorial Hopf algebras $(H, \xi) \to (\Sh, \xi_S)$ and is the unique such map.
        \item The map $\Psi$ is given as follows: for $h \in H_n$, we have \[
            \Psi(h) = \sum_{\a \comp n} \left( \xi^{\tensor \ell(\a)} \Delta_\a(h) \right) x_\a.
        \]
    \end{enumerate}
\end{thm}

\begin{proof}
    For any such coalgebra map $\Psi,$ consider the dual algebra map $\Psi^* \colon \Sh^\circ \to H^\circ.$ For all homogeneous $h \in H_n$ of positive degree, $\Psi^*$ must satisfy \begin{equation} \label{eq:Sh-universal-proof}
        \angle{\Psi^*(x_n^*), h} = \angle{x_n^*, \Psi(h)} = \xi_S(\Psi(h)) = \xi(h).
    \end{equation} Because $\Psi^*$ is graded, this condition uniquely determines the elements $\Psi^*(x_n^*) \in H^\circ.$ Dualizing \eqref{eq:shuffle-deconcatenate}, we see that the dual basis $\{x_\a^*\} \subseteq \Sh^\circ$ multiplies via $x_\a^* x_\b^* = x_{\a\b}^*,$ so $\Sh^\circ$ is freely generated as an algebra by the elements $\{x_n^* : n \in \Zpos\}.$ Thus there is a unique algebra map $\Psi^*$ satisfying these conditions, hence also a unique coalgebra map $\Psi.$ This proves (i).

    For (ii), it suffices to show that $\Psi \circ m_H = m_\Sh \circ (\Psi \tensor \Psi).$ Applying (i) to $(H \tensor H, \xi \circ m_H),$ which is an infinitorial Hopf algebra by Lemma~\ref{lem:infinitorial-tensor}, we get that there is a unique coalgebra map $\overline{\Psi} \colon H \tensor H \to \Sh$ satisfying $\xi \circ m_H = \xi_S \circ \overline{\Psi}$.     But $\Psi \circ m_H$ and $m_\Sh \circ (\Psi \tensor \Psi)$ are both coalgebra maps with this property: \[
        \xi_S \circ (\Psi \circ m_H) = (\xi_S \circ \Psi) \circ m_H = \xi \circ m_H
    \] and \A
        \xi_S \circ \left(m_\Sh \circ (\Psi \tensor \Psi)\right)
        &= (\xi_S \circ m_\Sh) \circ (\Psi \tensor \Psi) \\
        &= m_\bk \circ (\eps_\Sh \tensor \xi_S + \xi_S \tensor \eps_\Sh) \circ (\Psi \tensor \Psi) \\
        &= m_\bk \circ (\eps_H \tensor \xi + \xi \tensor \eps_H) \\
        &= \xi \circ m_H.
    \B Thus $\Psi \circ m_H = \overline{\Psi} = m_\Sh \circ (\Psi \tensor \Psi),$ proving (ii).

    For (iii), fix $h \in H_n$ with $n > 0$ and a composition $\a = (\a_1, \dots, \a_\ell) \comp n.$ Then by duality, \A
        \angle{x_\a^*, \Psi(h)}
        &= \angle{\Psi^*(x_\a^*), h} \\
                &= \angle{\Psi^*(x_{\a_1}^*) \tensor \dotsm \tensor \Psi^*(x_{\a_\ell}^*), \Delta^{\ell-1} (h) } \\
        &= \angle{\Psi^*(x_{\a_1}^*) \tensor \dotsm \tensor \Psi^*(x_{\a_\ell}^*), \Delta_\a (h) } \\
        &= \xi^{\tensor \ell(\a)} \Delta_\a(h),
    \B where we used \eqref{eq:Sh-universal-proof} and the fact that $\Psi^*(x_{\a_i})$ is homogeneous of degree $\a_i.$ Therefore, \A
        \Psi(h) = \sum_{\a \comp n} \angle{x_\a^*, \Psi(h)} \, x_\a = \sum_{\a \comp n} \left(\xi^{\tensor \ell(\a)} \Delta_\a(h)\right) x_\a,
    \B as desired.
\end{proof}

Just as Proposition~\ref{prop:abs-universal}(iii) allows one to assign a quasisymmetric invariant to an element of a combinatorial Hopf algebra, Theorem~\ref{thm:Sh-universal}(iii) allows one to assign a ``shuffle invariant'' to an element of an infinitorial Hopf algebra. As we will see, one way to think of a shuffle invariant is as a quasisymmetric invariant expressed in terms of a shuffle basis of $\QSym$.

\subsection{Bijections between characters and infinitesimal characters} \label{sec:bijections-char-infchar}

Let $H$ be a connected graded Hopf algebra. Proposition~\ref{prop:abs-universal} shows that for any character $\zeta \in \bX(H)$, there is a unique graded Hopf map $\Phi \colon H \to \QSym$ that satisfies $\zeta = \zeta_Q \circ \Phi.$ Conversely, any graded Hopf map $\Phi \colon H \to \QSym$ corresponds to a character $\zeta$ of $H$. In other words, there is a (canonical) bijection between characters of $H$ and graded Hopf maps $H \to \QSym.$

Theorem~\ref{thm:Sh-universal} provides an analogous result: for every infinitesimal character $\xi \in \Xi(H),$ there is a unique graded Hopf map $\Psi \colon H \to \Sh$ satisfying $\xi = \xi_S \circ \Psi,$ and conversely, every graded Hopf map $\Psi \colon H \to \Sh$ corresponds to an infinitesimal character of $H.$ Hence, there is a (canonical) bijection between infinitesimal characters of $H$ and graded Hopf maps $H \to \Sh.$

It follows then that isomorphisms of graded Hopf algebras between $\QSym$ and $\Sh$ give rise to bijections between characters and infinitesimal characters of $H.$ (Such isomorphisms exist in characteristic $0$ but not in positive characteristic, as we will demonstrate in Sections~\ref{sec:shuffle-bases} and \ref{sec:shuffle-qps}.) Given an isomorphism $\ph \colon \Sh \to \QSym,$ the induced bijection $\Xi(H) \to \bX(H)$ is the mapping $\xi \mapsto \zeta$ that makes the following diagram commute:

\[\begin{tikzcd}
    \bk && H && \bk \\
    \\
    & \Sh && \QSym
    \arrow["\zeta", from=1-3, to=1-5]
    \arrow["\xi"', from=1-3, to=1-1]
    \arrow["{\zeta_Q}"', from=3-4, to=1-5]
    \arrow["{\xi_S}", from=3-2, to=1-1]
    \arrow["\ph", shift left=1, from=3-2, to=3-4]
    \arrow["{\ph^{-1}}", shift left=1, from=3-4, to=3-2]
    \arrow["\Phi", from=1-3, to=3-4]
    \arrow["\Psi"', from=1-3, to=3-2]
\end{tikzcd}\]

In other words, under this bijection, an infinitesimal character $\xi \in \Xi(H)$ corresponds to the character $\zeta = \zeta_Q \circ \ph \circ \Psi,$ where $\Psi$ is the unique map of infinitorial Hopf algebras $(H, \xi) \to (\Sh, \xi_S).$ Conversely, $\zeta \in \bX(H)$ corresponds to $\xi = \xi_S \circ \ph^{-1} \circ \Phi,$ where $\Phi$ is the unique map of combinatorial Hopf algebras $(H, \zeta) \to (\QSym, \zeta_Q).$

\begin{example}
    Let $\mathcal{G}$ be as in Example~\ref{example-graphs1}. If $\zeta \in \bX(\mathcal{G})$ is defined by \[
        \zeta(G) = \begin{cases}
            1 & \text{if } G \text{ has no edges,} \\
            0 & \text{otherwise,}
        \end{cases}
    \] 
    then the map of combinatorial Hopf algebras $\Phi \colon (\mathcal{G}, \zeta) \to (\QSym, \zeta_Q)$ sends a graph $G$ to its \emph{chromatic symmetric function} $X_G$ (see \cite[Example 4.5]{abs}). When $G$ has $n$ vertices, we can express this symmetric function in terms of the power sum basis as $X_G = \sum_{\lambda \vDash n} c_\lambda p_\lambda$ for some constants $c_\lambda$.
    
    Let $\ph\colon \Sh \to \QSym$ be an isomorphism. Then $\ph$ is the unique map of combinatorial Hopf algebras $(\Sh, \zeta_Q \circ \ph) \to (\QSym, \zeta_Q),$ so by Proposition~\ref{prop:abs-universal}(iii), $\ph(x_n)$ must be a scalar multiple of $M_n = p_n$. For simplicity, assume that $\ph(x_n) = p_n$. Then $\xi_S \circ \ph^{-1}(p_n) = \xi_S(x_n) = 1$. On the other hand, when $\ell(\lambda) > 1$, $p_\lambda$ is a product of positive degree elements, so $\xi_S \circ \ph^{-1}(p_\lambda) = 0$.

    We now compute the infinitesimal character $\xi \in \Xi(\mathcal{G})$ that corresponds to $\zeta$ under the bijection $\bX(\mathcal{G}) \to \Xi(\mathcal{G})$ induced by $\ph.$ By the discussion above, 
    \[\xi(G) = \xi_S \circ \ph^{-1}(X_G) = \sum_{\lambda \comp n} c_\lambda \cdot \xi_S \circ \ph^{-1}(p_\lambda) = c_n.\]
                        Now, recall that $X_G$ is related to the chromatic polynomial $\chi_G$ of $G$ via \[
        \chi_G(k) = X_G(1^k) = X_G(\underbrace{1, \dots, 1}_{k}, 0, \dots).
    \] We have that $p_\lambda(1^k) = k^{\ell(\lambda)},$ so \[
        \chi_G(k) = \sum_{\lambda \comp n} c_\lambda p_\lambda(1^k) = \sum_{\lambda \comp n} c_\lambda k^{\ell(\lambda)}.
    \] Because $(n)$ is the only partition of $n$ of length $1,$ it follows that $c_n$ is the coefficient of $k^1$ in $\chi_G(k).$ In other words, $\xi$ is exactly the infinitesimal character of Example~\ref{example-graphs1}.
\end{example}

\begin{remark}
    The above example is exceptional in that the infinitesimal character $\xi$ was (up to scaling) independent of the choice of isomorphism $\ph.$ This occurs here because $X_G$ is a symmetric function, so it admits an expansion in terms of the symmetric power sums $\{p_\lambda\}.$ In view of Proposition~\ref{prop:abs-universal}, the fact that $X_G$ is symmetric arises from the cocommutativity of $\mathcal{G},$ which implies that $\zeta^{\tensor \ell(\a)} \Delta_\a(G) = \zeta^{\tensor \ell(\b)} \Delta_\b(G)$ whenever $\a \sim \b.$
\end{remark}

In Section~\ref{sec:qps-properties}, we will use the ideas of this section to recontextualize Example~\ref{ex:infinitorial-morphism}, as well as recover the \emph{exponential map} $\exp \colon \Xi(H) \to \bX(H).$

\section{Shuffle bases} \label{sec:shuffle-bases}

We now proceed to study the isomorphisms of graded Hopf algebras $\Sh \to \QSym.$ Such an isomorphism can be specified by a particular kind of basis of $\QSym$ which we call a shuffle basis.

\begin{defn} \label{shuffle-basis}
    A \emph{shuffle basis} of QSym is a graded basis $\{X_\a\}$ that satisfies
    \begin{equation} \label{eq:shuffle-basis-mult}
        X_\a X_\b = \sum_{\y \in \a \sh \b} X_\y
    \end{equation} for all compositions $\a$ and $\b,$ and
    \begin{equation} \label{eq:shuffle-basis-comult}
        \Delta(X_\y) = \sum_{\a\b = \y} X_\a \tensor X_\b
    \end{equation} for all compositions $\y.$
\end{defn}

A choice of shuffle basis $\{X_\a\}$ is then equivalent to a choice of isomorphism $\ph \colon \Sh \to \QSym$, where $\ph(x_\a) = X_\a$. In this section, we characterize shuffle bases by giving formulas for the change of basis between a shuffle basis and the monomial basis $\{M_\a\}.$

\subsection{Deconcatenation bases}

We begin by considering the comultiplicative structure of shuffle bases.

\begin{defn}
    A \emph{deconcatenation basis} of $\QSym$ is a graded basis $\{D_\a\}$ that satisfies \[
        \Delta(D_\y) = \sum_{\a\b = \y} D_\a \tensor D_\b
    \] for all compositions $\y.$
\end{defn}

Thus, by \eqref{eq:shuffle-basis-comult}, shuffle bases are deconcatenation bases, as is the monomial basis by \eqref{eq:monomial-deconcatenate}. Note that the dual basis $\{D_\a^*\}$ of a deconcatenation basis multiplies via $D_\a^* D_\b^* = D_{\a\b}^*.$

We call a function $f \colon \Compne \to \bk$ \emph{nonsingular} if $f(n) \neq 0$ for $n \in \Zpos.$ As we now show, changes of basis between deconcatenation bases can be expressed in terms of such functions. (Recall the definition of $f(\alpha, \beta)$ from Section~\ref{sec:compfunctions}.)

\begin{prop} \label{deconcat}
    Let $\{R_\a\}$ be a deconcatenation basis. Choose elements $Q_\a$ of $\QSym$ with $Q_\a \in \QSym_n$ for each $\a \comp n.$ Then $\{Q_\a\}$ is a deconcatenation basis if and only if there exists a nonsingular function $g \colon \Compne \to \bk$ such that, for all $\a$, \begin{equation} \label{eq:deconcatenate-basechange}
        R_\a = \sum_{\b \ge \a} g(\a, \b) Q_\b.
    \end{equation}
\end{prop}

\begin{proof}
    Note that if \eqref{eq:deconcatenate-basechange} holds, then by triangularity, $\{Q_\alpha\}$ is a basis if and only if $g(\alpha, \alpha) = \prod_{i=1}^{\ell(\alpha)} g(\alpha_i)$ is nonzero for all $\alpha$, which is equivalent to $g$ being nonsingular. Thus we may assume that $\{Q_\alpha\}$ is a basis. We then have that $\{Q_\alpha\}$ is a deconcatenation basis if and only if its dual basis $\{Q_\alpha^*\}$ is multiplicative, that is, $Q_{\b}^* = Q_{\b_1}^* \cdots Q_{\b_{\ell}}^*$ for all $\b$. This holds if and only if both sides agree when applied to $R_\a$ for all $\a$. In other words, we need that $\langle Q_\b^*, R_\a\rangle$ (the coefficient of $Q_\b$ in $R_\a$) equals
    \[\langle Q_{\beta_1}^* \cdots Q_{\beta_\ell}^*, R_\alpha\rangle = \langle Q_{\beta_1}^* \otimes \cdots \otimes Q_{\beta_\ell}^*, \Delta_\beta(R_\alpha) \rangle. \]
    But since $\{R_\alpha\}$ is a deconcatenation basis, the right hand side vanishes unless $\beta \geq \alpha.$ And if $\beta \ge \alpha,$ then $\Delta_\b(R_\a) = R_{\a^{(1)}}\otimes \dots \otimes R_{\a^{(\ell)}}$ implies 
    \[\langle Q_{\beta_1}^* \otimes \cdots \otimes Q_{\beta_\ell}^*, \Delta_\beta(R_\alpha) \rangle = \prod_{i=1}^{\ell} \langle Q^*_{\beta_i}, R_{\a^{(i)}} \rangle = \prod_{i=1}^{\ell} g(\a^{(i)}) = g(\a,\b),\]
    where we define $g \colon  \Compne \to \bk$ by $g(\a) = \angle{Q_n^*, R_\a}$ for $\a \comp n$. The result follows.
\end{proof}

In general, we can consider pairs of graded bases of $\QSym$ related by \eqref{eq:deconcatenate-basechange}. It turns out that the inverse change of basis always has the same form, regardless of whether the bases are deconcatenation bases.

\begin{prop} \label{prop:fg-equivalence}
    Let $\{Q_\a\}$ and $\{R_\a\}$ be graded bases of $\QSym.$ Then the following are equivalent:

    \begin{enumerate}[(i)]
        \item There exists a nonsingular function $f \colon \Compne \to \bk$ such that $Q_\a = \sum_{\b \ge \a} f(\a, \b) R_\b.$
        \item There exists a nonsingular function $g \colon \Compne \to \bk$ such that $R_\a = \sum_{\b \ge \a} g(\a, \b) Q_\b.$
    \end{enumerate}
\end{prop}

\begin{proof}
    By symmetry, it suffices to prove only one direction, so assume that (i) holds. To define the function $g,$ consider the equations \begin{equation} \label{eq:fg-equivalence}
        \sum_{\b \ge \a} f(\a, \b) g(\b) = \begin{cases}
            1 & \text{if } \ell(\a) = 1,\\
            0 & \text{otherwise}
        \end{cases}
    \end{equation} for $\a \in \Compne.$ By nonsingularity, $f(\a, \a) \neq 0$ for each $\a,$ so it follows from triangularity that \eqref{eq:fg-equivalence} uniquely determines $g \colon \Compne \to \bk.$ Additionally, when $\a = (n)$ we have $f(n) g(n) = 1,$ so $g(n) \neq 0$ and $g$ is also nonsingular. 
    
    To show that $g$ satisfies (ii), define $\til{R}_\a = \sum_{\b \ge \a} g(\a, \b) Q_\b$. Then for fixed $\a \in \Compne$, \[\begin{aligned}
        \sum_{\b \ge \a} f(\a,\b) \til{R}_\b
        &= \sum_{\b \ge \a} f(\a,\b) \sum_{\y \ge \b} g(\b, \y) Q_\y \\
        &= \sum_{\y \ge \a} \left(\sum_{\b : \, \a \le \b \le \y} f(\a,\b) g(\b,\y)\right) Q_\y.
    \end{aligned}\] For each $\y = (\y_1, \dots, \y_\ell) \ge \a,$ let $\a = \a^{(1)} \dotsm \a^{(\ell)}$ such that $\a^{(i)} \comp \y_i.$ Then the compositions $\b$ with $\a \le \b \le \y$ are exactly those of the form $\b = \b^{(1)} \dotsm \b^{(\ell)}$ where $\a^{(i)} \le \b^{(i)}.$ Hence, the sum in parentheses factors as \[\begin{aligned}
        \sum_{\b : \, \a \le \b \le \y} f(\a,\b) g(\b,\y)
        &= \prod_{i=1}^\ell \sum_{\b^{(i)} \ge \a^{(i)}} f(\a^{(i)}, \b^{(i)}) \, g(\b^{(i)}).
    \end{aligned}\] By \eqref{eq:fg-equivalence}, this product vanishes unless $\a^{(i)} = (\y_i)$ for each $i,$ or equivalently $\a = \y,$ in which case it equals $1.$ Therefore, we obtain \[
        \sum_{\b \ge \a} f(\a,\b) \til{R}_\b = Q_\y.
    \] Combined with (i), this forces $\til{R}_\a = R_\a$ (again by triangularity), completing the proof.
\end{proof}

From the above two results, it follows that a graded basis $\{X_\a\}$ is a deconcatenation basis if and only if there are nonsingular functions $f, g \colon \Compne \to \bk$ such that \begin{equation} \label{eq:shuffle-fg}
    X_\a = \sum_{\b \ge \a} f(\a, \b) M_\b
    \qquad
    \text{and}
    \qquad
    M_\a = \sum_{\b \ge \a} g(\a, \b) X_\b.
\end{equation}
Furthermore, the proof of Proposition~\ref{prop:fg-equivalence} shows that $g$ is uniquely determined by $f$ via \eqref{eq:fg-equivalence}, and vice versa with the roles of $f$ and $g$ switched.

It remains to determine what additional conditions on the functions $f$ and $g$ produce the multiplicative structure of $\{X_\a\}.$ To do this, we will need the universal properties of $\QSym$ (Proposition~\ref{prop:abs-universal}) and $\Sh$ (Theorem~\ref{thm:Sh-universal}).

\subsection{Characterizing shuffle bases}

Given a function $f \colon \Compne \to \bk,$ we define $f^\Sh$ to be the linear function $\Sh \to \bk$ given by $f^\Sh(x_\empty) = 1$ and $f^\Sh(x_\a) = f(\a)$ for $\a \neq \empty.$ Also, given $g \colon \Compne \to \bk,$ we define $g^\QSym$ to be the linear function $\QSym \to \bk$ given by $g^\QSym(M_\empty) = 0$ and $g^\QSym(M_\a) = g(\a)$ for $\a \neq \empty.$ 
\begin{thm} \label{thm:characterize-shuffle-bases}
    Let $\{X_\a\}$ be a deconcatenation basis, and let $f, g \colon \Compne \to \bk$ be nonsingular functions satisfying \eqref{eq:shuffle-fg}.
                Then the following are equivalent:
    \begin{enumerate}[(i)]
        \item $\{X_\a\}$ is a shuffle basis.

        \item $f^\Sh \in \mathbb{X}(\Sh).$
        
        \item $g^\QSym \in \Xi(\QSym).$
    \end{enumerate}
\end{thm}

\begin{proof}
    Let $\ph\colon \Sh \to \QSym$ be the linear map defined by $\ph(x_\a) = X_\a.$  Recall that $\{X_\a\}$ is a shuffle basis if and only if $\ph$ is a Hopf isomorphism. Since $\{X_\a\}$ is a deconcatenation basis, $\ph$ is automatically a coalgebra isomorphism. 

    For $\a \neq \empty,$ we have $\ph(x_\a) = X_\a = \sum_{\b \ge \a} f(\a, \b) M_\b,$ so
    \[
        \zeta_Q \circ \ph(x_\a) = \sum_{\b \ge \a} f(\a, \b) \zeta_Q(M_\b) = f(\a) = f^\Sh(x_\a).
    \] This also holds for $\a = \empty$ since $f(\empty) = 1$; thus, 
    \begin{equation} \label{eq:Sh-equivalence-fSh}
        f^\Sh = \zeta_Q \circ \ph.
    \end{equation}
    Similarly, for $\a \neq \empty,$ we have $\ph^{-1}(M_\a) = \ph^{-1}\left(\sum_{\b \ge \a} g(\a,\b) X_\b\right) = \sum_{\b \ge \a} g(\a,\b) x_\b,$ so
    \[
        \xi_S \circ \ph^{-1}(M_\a) = \sum_{\b \ge \a} g(\a,\b) \xi_S(x_\b) = g(\a) = g^{\QSym}(M_\a).
    \] This also holds for $\a = \empty$ since $g(\empty) = 0$; thus,
    \begin{equation} \label{eq:Sh-equivalence-gQSym}
        g^\QSym = \xi_S \circ \ph^{-1}.
    \end{equation}
    
    Now suppose that $\{X_\a\}$ is a shuffle basis. Then $\ph$ and $\ph^{-1}$ are algebra maps, so it follows from \eqref{eq:Sh-equivalence-fSh} and \eqref{eq:Sh-equivalence-gQSym} that $f^\Sh \in \bX(\Sh)$ and $g^\QSym \in \Xi(\QSym).$ This shows that (i) implies (ii) and (iii).
    
    Suppose that (ii) holds. Then $(\Sh, f^\Sh)$ is a combinatorial Hopf algebra, so by \eqref{eq:Sh-equivalence-fSh} and Proposition~\ref{prop:abs-universal}, $\ph$ is the unique map of combinatorial Hopf algebras $(\Sh, f^\Sh) \to (\QSym, \zeta_Q)$. Therefore $\{X_\a\}$ is a shuffle basis. This shows that (ii) implies (i).

    Similarly, suppose that (iii) holds. Then $(\QSym, g^\QSym)$ is an infinitorial Hopf algebra, so by \eqref{eq:Sh-equivalence-gQSym} and Theorem~\ref{thm:Sh-universal}, $\ph^{-1}$ is the unique map of infinitorial Hopf algebras $(\QSym, g^\QSym) \to (\Sh, \xi_S).$ Therefore $\{X_\a\}$ is again a shuffle basis. This shows that (iii) implies (i), completing the proof.
\end{proof}

In view of this result, we make the following definitions:

\begin{defn} \leavevmode
    \begin{enumerate}[(i)]
        \item A \emph{shuffle character} is a map $f \colon \Compne \to \bk$ such that $f^\Sh \in \bX(\Sh).$
        \item A \emph{quasisymmetric infinitesimal character} is a map $g \colon \Compne \to \bk$ such that $g^\QSym \in \Xi(\QSym).$
    \end{enumerate}
\end{defn}

Equivalently, $f$ is a shuffle character if and only if $f(\a) f(\b) = \sum_{\y \in \a \sh \b} f(\y)$ for all compositions $\a$ and $\b,$ where we define $f(\empty)$  to be $1.$ (One can likewise describe an analogous condition for quasisymmetric infinitesimal characters, but we will not need this.) We have thus shown that the following objects are in bijection:

\begin{itemize}
    \item Isomorphisms of graded Hopf algebras $\Sh \to \QSym$;
    \item Shuffle bases $\{X_\a\}$ of $\QSym$;
    \item Nonsingular shuffle characters $f \colon \Compne \to \bk$;
    \item Nonsingular quasisymmetric infinitesimal characters $g \colon \Compne \to \bk$.
\end{itemize}

\begin{remark}
    There do not exist nonsingular shuffle characters in positive characteristic. Indeed, if $\bk$ has characteristic $p > 0,$ then \[
        f(1)^p = \sum_{\a \in \underbrace{\scriptstyle 1 \sh \dotsm \sh 1}_{p}} f(\a) = p! f(\underbrace{1 \ldots 1}_p) = 0,
    \] which forces $f(1) = 0.$ (Similarly, $f(n) = 0$ for \emph{all} $n \in \Zpos.$) It follows that $\Sh$ and $\QSym$ are not isomorphic in positive characteristic.
\end{remark}

\section{Shuffle characters and quasisymmetric power sums} \label{sec:shuffle-qps}

In this section, we introduce quasisymmetric power sums and show how they are closely related to shuffle bases. Then, using the results of the previous section, we present constructions of quasisymmetric power sums which recover several bases found in the literature. For the rest of this paper, we assume that $\bk$ has characteristic zero.

The definition of a quasisymmetric power sum comes from \cite[Proposition 4.1]{awvw} as a result of dualizing the definition of a \emph{noncommutative} power sum. These noncommutative power sums, in turn, are defined by lifting the algebraic properties of the symmetric power sums $\{p_\lambda\}$ from $\Sym$ to the algebra of noncommutative symmetric functions $\NSym,$ which is the graded dual of $\QSym.$ (We are largely able to avoid the use of duality in this paper; as such, we give the following definition without any reference to $\NSym.$)

\begin{defn} \label{def:qps}
    A \emph{quasisymmetric power sum} (QPS) basis is a graded basis $\{P_\a\}$ of $\QSym$ that satisfies the following conditions:
    
    \begin{enumerate}[(i)]
        \item (Multiplication) $P_\a P_\b = \frac{z_\a z_\b}{z_{\a\b}} \sum_{\y \in \a \sh \b} P_\y$ for all compositions $\a, \b$;

        \item (Comultiplication) $\Delta(P_\a) = \sum_{\b\y = \a} \frac{z_\a}{z_\b z_\y} P_\b \tensor P_\y$ for all compositions $\a$;

        \item (Power sum refinement) $\sum_{\a \sim \lambda} P_\a = p_\lambda$ for every partition $\lambda.$
    \end{enumerate}
\end{defn}

The first two conditions of Definition~\ref{def:qps} are, up to scaling, exactly those that define a shuffle basis. However, as we will now show, the third condition comes for free under the correct choice of scaling.

\begin{defn} \leavevmode
    \begin{enumerate}[(i)]
        \item A function $f \colon \Compne \to \bk$ is \emph{normalized} if $f(n) = 1$ for all $n \in \Zpos.$
        
        \item A shuffle basis $\{X_\a\}$ is \emph{normalized} if $X_n = M_n$ for all $n \in \Zpos.$
    \end{enumerate}
\end{defn}

It is easy to check that a shuffle basis $\{X_\a\}$ is normalized if and only if the corresponding shuffle character $f$ is normalized, if and only if the corresponding quasisymmetric infinitesimal character $g$ is normalized.

The next result states that QPS bases are, in fact, equivalent to normalized shuffle bases up to a consistent choice of scaling.

\begin{prop}
    A graded basis $\{P_\a\}$ of $\QSym$ is a QPS basis if and only if the basis $\{P_\a'\}$ given by $P_\a' = \frac{1}{\aut(\a)} P_\a$ is a normalized shuffle basis.
\end{prop}

\begin{proof}
    Because $z_\a = \aut(\a) \p(\a),$ it follows that $\f{z_\a z_\b}{z_{\a\b}} = \f{\aut(\a) \aut(\b)}{\aut(\a\b)}$ for all $\a$ and $\b.$ Thus, after rearranging, we find that conditions (i) and (ii) of Definition~\ref{def:qps} are equivalent to $\{P_\a'\}$ being a shuffle basis.
    
    It remains to show that condition (iii) is equivalent to $\{P_\a'\}$ being normalized. Taking $\lambda = (n)$ in condition (iii), we find that $P_n' = P_n = p_n = M_n,$ so $\{P_\a'\}$ is indeed normalized. Conversely, suppose that $\{P_\a'\}$ is a normalized shuffle basis. Then for any partition $\lambda$ of length $\ell$,
    \[p_\lambda = P'_{\lambda_1} \cdots P'_{\lambda_\ell} = \sum_{\alpha \in \lambda_1 \sh \cdots \sh \lambda_\ell} P'_{\a} = \aut(\a) \sum_{\alpha \sim \lambda} P'_\a = \sum_{\alpha \sim\lambda} P_\a,\]
    so (iii) is satisfied.
\end{proof}

Thus, every QPS basis $\{P_\a\}$ satisfies \begin{equation} \label{eq:qps-shuffle-char}
    P_\a = \aut(\a) \sum_{\b \ge \a} f(\a, \b) M_\b
\end{equation} where $f$ is a normalized shuffle character, and \begin{equation} \label{eq:qps-qic}
    M_\a = \sum_{\b \ge \a} g(\a, \b) \f{1}{\aut(\b)} P_\b
\end{equation} where $g$ is a normalized quasisymmetric infinitesimal character. Conversely, any choice of $f$ or choice of $g$ produces a QPS basis by the above formulas.

Furthermore, let $f$ be a shuffle character (not necessarily normalized). If $h \colon \Zpos \to \bk$ is any map, then one can check that $\tilde{f}(\a) = \prod_{i=1}^{\ell(\a)} h(\a_i) \cdot f(\a)$ also defines a shuffle character. In particular, if $f$ is nonsingular, then it can be normalized by taking $h(n) = \f{1}{f(n)}.$

This leads to a general method for constructing quasisymmetric power sums: construct a (nonsingular) shuffle character $f$, normalize it as above, then use \eqref{eq:qps-shuffle-char}.\footnote{
    Using \eqref{eq:qps-qic}, one can also construct QPS bases by constructing quasisymmetric infinitesimal characters and normalizing them. However, we do not yet know of any general such constructions.
} Using this method, we now present two general constructions of shuffle characters, which we use to recover four QPS bases from the literature.

\subsection{The prefix sum construction} \label{prefix-sum}

\begin{thm} \label{thm:prefix-sum}
    Let $\tau \colon \Zpos \to \bk$ be any function. Then the map $f \colon \Compne \to \bk$ defined by \[
        f(\a) = \prod_{i=1}^{\ell(\a)} \left(\tau(\a_1) + \dotsm + \tau(\a_i)\right)^{-1}
    \] is a shuffle character, assuming that the sums $\tau(\a_1) + \dots + \tau(\a_i)$ are all nonzero.
\end{thm}

\begin{proof}
    We prove that $f(\a) f(\b) = \sum_{\y \in \a \sh \b} f(\y)$ by induction on $\ell(\a) + \ell(\b).$ The base cases $\a=\empty$ and $\b=\empty$ are trivial. Next, fix nonempty compositions $\a = (\a_1, \dots, \a_k)$ and $\b = (\b_1, \dots, \b_\ell),$ and let $\a' = (\a_1, \dots, \a_{k-1})$ and $\b' = (\b_1, \dots, \b_{\ell-1}).$ Every shuffle $\y \in \a \sh \b$ is either a concatenation $\delta\a_k$ for $\delta \in \a' \sh \b,$ or a concatenation $\delta\b_\ell$ for $\delta \in \a \sh \b'$. Thus, letting $s = \tau(\a_1) + \dots + \tau(\a_k)$ and $t = \tau(\b_1) + \dots + \tau(b_\ell)$, we have \A
        \sum_{\y \in \a \sh \b} f(\y)
        = \sum_{\delta \in \a' \sh \b} f(\delta \a_k) + \sum_{\delta \in \a \sh \b'} f(\delta \b_\ell)
        &= (s + t)^{-1} 
            \left( \sum_{\delta \in \a' \sh \b} f(\delta) + \sum_{\delta \in \a \sh \b'} f(\delta) \right) \\
        &= (s + t)^{-1} \left(f(\a') f(\b) + f(\a) f(\b')\right) \\
        &= (s + t)^{-1} \left(s f(\a) f(\b) + t f(\a) f(\b)\right) \\
        &= f(\a)f(\b),
    \B as desired.
\end{proof}

Using the prefix sum construction, we can recover the type I and type II quasisymmetric power sums as described below. The duals of these bases were introduced in \cite{gelfand}, and these bases were the main objects of study in \cite{bdhmn}, where Ballantine, Daugherty, Hicks, Mason, and Niese gave combinatorial proofs that they indeed satisfy Definition~\ref{def:qps}.

\begin{example} \label{ex:typeI}
    Taking $\tau(n) = n$ in Theorem~\ref{thm:prefix-sum}, we obtain the shuffle character \[
        f(\a)
        = \prod_{i=1}^{\ell(\a)} (\a_1 + \dots + \a_i)^{-1} = \f{1}{\pi(\a)},
    \] where we define $\pi(\a) = \a_1 (\a_1 + \a_2) \dotsm (\a_1 + \dots + \a_\ell).$ We have $f(n) = \frac{1}{n}$ for $n \in \Zpos,$ so $f$ normalizes to $\tilde{f}(\a) = \prod_{i=1}^{\ell(\a)} \a_i \cdot f(\a) = \f{\p(\a)}{\pi(\a)}.$ Hence, the corresponding quasisymmetric power sum is \[
        \Psi_\a = \aut(\a) \sum_{\b \ge \a} \f{\p(\a,\b)}{\pi(\a, \b)} M_\b = z_\a \sum_{\b \ge \a} \f{1}{\pi(\a,\b)} M_\b
    \] (since $\p(\a,\b) = \p(\a)$). These are the \emph{type I quasisymmetric power sums} $\{\Psi_\a\},$ as given in \cite[\S 3.1]{bdhmn}.
\end{example}

\begin{example} \label{ex:typeII}
    Taking $\tau(n) = 1$ in Theorem~\ref{thm:prefix-sum}, we obtain the shuffle character \[
        f(\a) = \prod_{i=1}^{\ell(\a)} i^{-1} = \f{1}{\ell(\a)!}.
    \]
    This $f$ is already normalized, so the corresponding quasisymmetric power sum is \[
        \Phi_\a = \aut(\a) \sum_{\b \ge \a} \f{1}{\ell^!(\a, \b)} M_\b.
    \] (Here $\ell^!$ means the function $\ell^!(\a) = \ell(\a)!.$) These are the \emph{type II quasisymmetric power sums} $\{\Phi_\a\},$ as given in \cite[\S 3.2]{bdhmn}.
\end{example}

In Sections~\ref{sec:typeI} and \ref{sec:typeII}, we consider the bijections between characters and infinitesimal characters induced by these two bases. As a consequence, we will rederive known formulas for the corresponding quasisymmetric infinitesimal characters.

\subsection{The ordered partition construction} \label{section-ordered-partition}

Next, we present a construction that ``blends together'' already constructed shuffle characters.

We define an \emph{ordered partition} of $\Zpos$ to be a pair $(\mathscr{C}, \prec),$ where $\mathscr{C}$ is a set partition of $\Zpos$ and $\prec$ is a total order on the elements of $\mathscr{C}.$ We say that a composition $\a = (\a_1, \dots, \a_\ell)$ \emph{respects} $(\mathscr{C}, \prec)$ if $C_1 \preceq \dotsm \preceq C_\ell,$ where $C_i$ is the (unique) element of $\mathscr{C}$ containing $\a_i.$ Finally, given a composition $\a$ and a set $C \subseteq \Zpos,$ define the composition $\a|_C$ to be the subsequence of $\a$ consisting of the parts of $\a$ that lie in $C,$ or $\empty$ if there are no such parts.

\begin{thm} \label{thm:ordered-partition}
    Let $(\mathscr{C}, \prec)$ be an ordered partition of $\Zpos.$ For each $C \in \mathscr{C},$ choose a shuffle character $f_C \colon \Compne \to \bk.$ Then the map $f \colon \Compne \to \bk$ defined by \begin{equation} \label{eq:ordered-partition}
        f(\a) = \begin{cases}
            \prod_{C \in \mathscr{C}} f_C\left(\a|_C\right) & \text{if } \a \text{ respects } (\mathscr{C}, \prec), \\
            0 & \text{otherwise}
        \end{cases}
    \end{equation} (and $f(\empty) = 1$) is a shuffle character.
\end{thm}

\begin{proof}
    Fix compositions $\a$ and $\b$; we prove that 
        $f(\a) f(\b) = \sum_{\y \in \a \sh \b} f(\y)$.
    If $\a$ does not respect $(\mathscr{C}, \prec),$ then no composition that contains $\a$ as a subsequence can respect $(\mathscr{C}, \prec)$ either, so both sides of this equation vanish, and similarly for $\b$. Therefore, we may assume that $\a$ and $\b$ both respect $(\mathscr{C}, \prec)$. Let $\a = \a|_{C_1} \dotsm \a|_{C_r}$ and $\b = \b|_{C_1} \dotsm \b|_{C_r}$ for some $C_1 \prec \dotsm \prec C_r \in \mathscr{C}.$ In this case, the shuffles $\y \in \a \sh \b$ that respect $(\mathscr{C}, \prec)$ are exactly those of the form $\y = \y^{(1)} \dotsm \y^{(r)},$ where $\y^{(i)} \in \a|_{C_i} \sh \b|_{C_i}$ for each $i.$ It follows that \A
        \sum_{\y \in \a \sh \b} f(\y)
        = \sum_{\y^{(1)}, \dots, \y^{(r)}} f_{C_1}(\y^{(1)}) \dotsm f_{C_r}(\y^{(r)})
        &= \prod_{i=1}^r \sum_{\y^{(i)}} f_{C_i}(\y^{(i)}) \\
        &= \prod_{i=1}^r f_{C_i}(\a|_{C_i}) f_{C_i}(\b|_{C_i}) \\
        &= f(\a) f(\b),
    \B as desired.
\end{proof}

\begin{remark}
    Note that in this construction, each shuffle character $f_C$ need only be defined for compositions whose parts are in $C.$ Conversely, if $f_C$ is only defined for such compositions, then it can be extended to a shuffle character on all of $\Comp$ by simply defining $f_C(\a) = 0$ for all other compositions.
\end{remark}

\begin{remark} \label{rmk:small-support}
    In a sense, the ordered partition construction is the only way to produce shuffle characters with such small support. Indeed, let $(\mathscr{C}, \prec)$ be an ordered partition of $\Zpos,$ and suppose $f$ is a shuffle character such that $f(\a) = 0$ unless $\a$ respects $(\mathscr{C}, \prec).$ Then \eqref{eq:ordered-partition} is satisfied simply by taking $f_C = f$ for each $C \in \mathscr{C}.$ To see this, suppose $\a$ respects $(\mathscr{C}, \prec),$ and let $\a = \a|_{C_1} \dotsm \a|_{C_r}$ for some $C_1 \prec \dotsm \prec C_r \in \mathscr{C}.$ Then \[
        f(\a|_{C_1}) \dotsm f(\a|_{C_r}) = \sum_{\b \in \a|_{C_1} \sh \dotsm \sh \a|_{C_r}} f(\b) = f(\a)
    \] as claimed, because $\a$ is the only shuffle of $\a|_{C_1}, \dots, \a|_{C_r}$ that respects $(\mathscr{C}, \prec).$
\end{remark}

Using the ordered partition construction, we can recover two more QPS bases from the literature, the even-odd and combinatorial quasisymmetric power sums.

\begin{example} \label{ex:even-odd-basis}
    In Theorem~\ref{thm:ordered-partition}, let the ordered partition $(\mathscr{C}, \prec)$ be given by $\mathbb{E} \prec \mathbb{O},$ where $\mathbb{E} = \{2, 4, 6, \dots\}$ and $\mathbb{O} = \{1, 3, 5, \dots\}.$ We take $f_\mathbb{E}(\a) = f_\mathbb{O}(\a) = \frac{1}{\ell(\a)!},$ the shuffle character from Example~\ref{ex:typeII}. This gives the shuffle character \[
        f(\a) = \begin{cases}
            \frac{1}{\even(\a)! \odd(\a)!} & \text{if } \a = \a|_\mathbb{E} \, \a|_\mathbb{O}, \\
            0 & \text{otherwise},
        \end{cases}
    \] where $\even(\a)$ and $\odd(\a)$ are the number of even parts and odd parts of $\a,$ respectively. Because $f$ is normalized, the corresponding quasisymmetric power sum is given by \[
        \mathfrak{e}_\a = \aut(\a) \sum_{\b \ge_{\text{eo}} \a} \frac{1}{\even^!(\a,\b) \odd^!(\a,\b)} M_\b,
    \] where $\a \le_{\text{eo}} \b$ if $\a \le \b$ and every $\a^{(i)} \vDash \b_i$ consists of even parts followed by odd parts.
    
    We call this basis the \emph{even-odd quasisymmetric power sums}. It was defined, up to scaling, by Aliniaeifard and Li \cite{al}. In that paper, $\{\mathfrak{e}_\a\}$ is shown to be an eigenbasis of a particular canonical Hopf map $\Theta \colon \QSym \to \QSym.$ In Section~\ref{sec:even-odd}, we exhibit a family of QPS bases (including $\{\mathfrak{e}_\a\}$) with the same property; we also study the quasisymmetric infinitesimal characters of these bases.
\end{example}

\begin{example}
    In Theorem~\ref{thm:ordered-partition}, let the ordered partition of $\Zpos$ be the reverse of the usual ordering on positive integers: $\dotsm \prec \{3\} \prec \{2\} \prec \{1\}.$ For each set $\{n\}$ in the partition, we again take the shuffle character $f_{\{n\}}(\a) = \frac{1}{\ell(\a)!}.$ With these choices, one can check that we obtain the shuffle character \[
        f(\a) = \begin{cases}
            \frac{1}{\aut(\a)} & \text{if } \a \text{ has weakly decreasing parts}, \\
            0 & \text{otherwise}.
        \end{cases}
    \] Because $f$ is normalized, the corresponding quasisymmetric power sum is \[
        \mathfrak{p}_\a = \aut(\a) \sum_{\a \le_{\text{wd}} \b} \frac{1}{\aut(\a, \b)} M_\b,
    \] where $\a \le_{\text{wd}} \b$ if $\a \le \b$ and every $\a^{(i)} \vDash \b_i$ has weakly decreasing parts. These are the \emph{combinatorial quasisymmetric power sums}, first constructed by Alinaeifard, Wang, and van Willigenburg~\cite{awvw} using $P$-partitions, and also studied by Lazzeroni~\cite{lazzeroni}.\footnote{
        While this formula for $\mathfrak{p}_\a$ does not appear explicitly in either of these sources, it can be derived starting from \cite[Theorem 5.8]{awvw} and \cite[\S 3, Definition 2]{lazzeroni}. Alternatively, the results of this paper are sufficient to show that the basis we have constructed must agree with the ones found in these sources. First, one can check that the bases found in these sources have shuffle characters with the same support as ours. Then the desired conclusion follows from an analysis using Remark~\ref{rmk:small-support}, or by using Theorem~\ref{thm:integral-qps} below.
    } As was established in \cite[\S 4.2, Theorem 4]{lazzeroni}, the dual basis of $\{\mathfrak{p}_\a\}$ is (up to scaling) the so-called ``Zassenhaus basis'' of $\NSym,$ defined in \cite[Definition 5.26]{krob-leclerc-thibon}.

    The combinatorial quasisymmetric power sums, unlike the others we have discussed, have nonnegative \emph{integer} coefficients when expanded into the monomial basis $\{M_\a\}.$ Indeed, one can check that $\frac{\aut(\a)}{\aut(\a,\b)}$ is a product of multinomial coefficients \begin{equation} \label{eq:aut-quotient}
            \frac{\aut(\a)}{\aut(\a,\b)} = \prod_{i \ge 1} \binom{m_i(\a)}{m_i(\a^{(1)}), \dots, m_i(\a^{(\ell(\b))})}.
    \end{equation} In Section~\ref{sec:qps-properties}, we use our characterization of quasisymmetric power sums to determine \emph{all} QPS bases with this integrality property.
\end{example}

\begin{remark} \label{rmk:reverse-combinatorial-qps}
    The authors of \cite{awvw} also construct a basis of ``reverse combinatorial power sums'' $\{\mathfrak{p}_\a^r\}$ that has the same integrality property. Using the definitions provided in that paper, one can check that for $\a, \b \in \Comp_n,$ the coefficient of $M_{\tiny \rev{\b}}$ in $\mathfrak{p}_{\tiny \rev{\a}}^r$ is equal to the coefficient of $M_\b$ in $\mathfrak{p}_\a.$ Hence, the basis $\{\mathfrak{p}_\a^r\}$ corresponds to the normalized shuffle character \[
        f(\a) = \begin{cases}
            \frac{1}{\aut(\a)} & \text{if } \a \text{ has weakly } \textit{increasing} \text{ parts}, \\
            0 & \text{otherwise}.
        \end{cases}
    \] In other words, it can be obtained by the ordered partition construction just as above, but instead using the partition $\{1\} \prec \{2\} \prec \{3\} \prec \dotsm.$
\end{remark}

\section{Properties of quasisymmetric power sums} \label{sec:qps-properties}

In this section, we use the framework we have developed to derive additional properties of the quasisymmetric power sum bases of Section~\ref{sec:shuffle-qps}.

\subsection{The type I basis} \label{sec:typeI}

Using the type I quasisymmetric power sums $\{\Psi_\a\}$, as well as some results from \cite{liuweselcouch}, we can complete the picture of Example~\ref{ex:infinitorial-morphism}. We reuse the notations $\mathcal{P},$ $\zeta,$ $\xi,$ and $\eta.$ Let \[
    \psi_\a = \f{\Psi_\a}{z_\a} = \sum_{\b \ge \a} \f{1}{\pi(\a, \b)} M_\b.
\] Note that $\{\psi_\a\}$ is a shuffle basis, since $\psi_\a = \f{1}{\p(\a)} \cdot \f{1}{\aut(\a)} \Psi_\a$ and $\{\f{1}{\aut(\a)} \Psi_\a\}$ is a (normalized) shuffle basis. The key fact that relates the type I basis to Example~\ref{ex:infinitorial-morphism} is \cite[Lemma 3.3]{liuweselcouch}, which gives \[
    \eta(\psi_\a) = \begin{cases}
        1 & \text{if } \ell(\a) = 1, \\
        0 & \text{otherwise}.
    \end{cases}
\] Thus, if $\ph \colon \Sh \to \QSym$ is the isomorphism given by $\ph(x_\a) = \psi_\a,$ then $\eta \circ \ph = \xi_S.$ Recall from Example~\ref{ex:infinitorial-morphism} that $\eta(K_P(\mathbf{x})) = \xi(P),$ and that $P \mapsto K_P(\mathbf{x})$ is the unique morphism of combinatorial Hopf algebras $(\mathcal{P}, \zeta) \to (\QSym, \zeta_Q).$ Putting all this together, we conclude that $\zeta$ and $\xi$ correspond to each other under the bijection $\bX(\mathcal{P}) \to \Xi(\mathcal{P})$ induced by $\ph,$ as shown below.

\[\begin{tikzcd}
    \bk && {\mathcal{P}} & {} & \bk \\
    \\
    & \Sh && \QSym
    \arrow["\zeta", from=1-3, to=1-5]
    \arrow["\xi"', from=1-3, to=1-1]
    \arrow[from=1-3, to=3-2]
    \arrow["{P \mapsto K_P(\mathbf{x})}"{description}, from=1-3, to=3-4]
    \arrow["\varphi"', from=3-2, to=3-4]
    \arrow["{\zeta_Q}"', from=3-4, to=1-5]
    \arrow["{\xi_S}", from=3-2, to=1-1]
    \arrow["\eta"{description, pos=0.7}, dashed, from=3-4, to=1-1]
\end{tikzcd}\]

We can now derive the quasisymmetric infinitesimal character $g$ of the type I basis. Indeed, suppose that $M_\a = \sum_{\b \ge \a} g(\a,\b) \f{1}{\aut(\b)} \Psi_\b = \sum_{\b \ge \a} g(\a,\b) \p(\b) \psi_\b.$ Applying $\eta$ gives \[
    (-1)^{\ell(\a)-1} \lp(\a) = \sum_{\b \ge \a} g(\a,\b) \p(\b) \, \eta(\psi_\b) = g(\a)\cdot |\a|
\] so we arrive at \[g(\a) = (-1)^{\ell(\a)-1} \f{\lp(\a)}{|\a|}.\] (One can also verify this formula by showing that \eqref{eq:fg-equivalence} holds, either as given or with the roles of $f$ and $g$ switched.)

\subsection{The type II basis} \label{sec:typeII}

We now show how to recover the exponential map $\exp \colon \Xi(H) \to \bX(H)$ using the type II quasisymmetric power sums. As a consequence, we also rederive the quasisymmetric infinitesimal character of this basis.

Let $X_\a = \f{1}{\aut(\a)} \Phi_\a = \sum_{\b \ge \a} \f{1}{\ell^!(\a,\b)} M_\b$ be the normalized shuffle basis element corresponding to the type II quasisymmetric power sums, and denote by $\ph \colon \Sh \to \QSym$ the isomorphism $\ph(x_\a) = X_\a.$ Let $H$ be a connected graded Hopf algebra, choose $\xi \in \Xi(H),$ and let $\zeta \in \bX(H)$ be the character that corresponds to $\xi$ under the bijection $\Xi(H) \to \bX(H)$ induced by $\ph.$ Then, Theorem~\ref{thm:Sh-universal}(iii) and the results of Section~\ref{sec:bijections-char-infchar} imply that \[
    \zeta(h) = \zeta_Q \circ \ph\left(\sum_{\a \comp n} \xi^{\tensor \ell(\a)} \Delta_\a(h) x_\a \right)
\] for $h \in H_n.$ Now $\zeta_Q \circ \ph(x_\a) = \zeta_Q(X_\a) = \f{1}{\ell(\a)!},$ so we get \[
    \zeta(h)
    = \sum_{\a \comp n} \xi^{\tensor \ell(\a)} \Delta_\a(h) \f{1}{\ell(\a)!}
    = \sum_{m \ge 0} \f{1}{m!} \sum_{\substack{\a \comp n \\ \ell(\a) = m}} \xi^{\tensor m} \Delta_\a(h)
    = \sum_{m \ge 0} \f{1}{m!} \xi^{*m}(h),
\] where $\xi^{*m}$ denotes the $m$th convolutional power of $\xi.$ (Here we used the fact that $\xi$ annihilates elements in degree zero, so that $\sum_{\a \comp n, \, \ell(\a) = m} \xi^{\tensor m} \Delta_\a(h) = \xi^{\tensor m} \Delta^{m-1}(h).$) In other words, in the algebra of linear maps $H \to \bk$ with the convolution product, we have $\zeta = \exp \xi.$ Thus $\ph$ induces the exponential map $\exp \colon \Xi(H) \to \bX(H).$

Using this observation, we can also rederive the quasisymmetric infinitesimal character $g$ of the type II basis. Suppose we instead fix a character $\zeta \in \bX(H).$ By a similar computation as above, the infinitesimal character $\xi$ that corresponds to $\zeta$ is given by \begin{equation} \label{eq:typeII-g-computation}
    \xi(h) = \sum_{\a \comp n} \zeta^{\tensor \ell(\a)} \Delta_\a(h) g(\a)
\end{equation} for $h \in H_n$ and $n>0.$ But $\zeta = \exp \xi,$ so $\xi = \log \zeta$, that is, \[\begin{aligned}
    \xi(h) &= \sum_{m \ge 1} \f{(-1)^{m-1}}{m} (\zeta - \eps)^{*m}(h) \\
    &= \sum_{m \ge 1} \f{(-1)^{m-1}}{m} \sum_{\substack{\a \comp n \\ \ell(\a) = m}} (\zeta - \eps)^{\tensor \ell(\a)} \Delta_\a(h) \\
    &= \sum_{\a \comp n} \f{(-1)^{\ell(\a)-1}}{\ell(\a)} \zeta^{\tensor \ell(\a)} \Delta_\a(h),
\end{aligned}\] for $h \in H_n$ and $n > 0.$ (Here we used the fact that $\zeta - \eps$ annihilates elements in degree zero and $\eps$ annihilates elements in positive degree.) Now \eqref{eq:typeII-g-computation} uniquely determines $g$ (for example, when $(H, \zeta) = (\QSym, \zeta_Q),$ one can check that $\xi = g^\QSym$), so we get \[
    g(\a) = \f{(-1)^{\ell(\a)-1}}{\ell(\a)}.
\] (Again, this formula for $g(\a)$ can also be verified directly by means of \eqref{eq:fg-equivalence}.)

\subsection{The even-odd basis} \label{sec:even-odd}

In \cite{al}, the even-odd basis $\{\mathfrak{e}_\a\}$ was proved to be an eigenbasis for a certain canonical Hopf map $\Theta \colon \QSym \to \QSym.$ Here, we generalize this result to a larger family of QPS bases and also consider their quasisymmetric infinitesimal characters.

Let $\nu_Q$ be the character of $\QSym$ defined by $\nu_Q = \overline{\zeta}_Q^{-1} * \zeta_Q,$ where $\overline{\zeta}_Q(M_\a) = (-1)^{|\a|} \zeta_Q(M_\a)$ and $\overline{\zeta}_Q^{-1}$ is its convolutional inverse. This is the canonical \emph{odd character} associated to $\zeta_Q,$ as in \cite[Example 1.3]{abs}. We will need the following formula for $\nu_Q,$ from \cite[(4.8)]{abs}: \begin{equation} \label{eq:nuQ_Malpha}
    \nu_Q(M_\a) = \begin{cases}
        1 & \text{if } \a = \empty, \\
        2 (-1)^{|\a| + \ell(\a)} & \text{if } \lp(\a) \text{ is odd}, \\
        0 & \text{otherwise.}
    \end{cases}
\end{equation} The map $\Theta$ is defined to be the unique morphism of combinatorial Hopf algebras \[
    \Theta \colon (\QSym, \nu_Q) \to (\QSym, \zeta_Q).
\] The image of $\Theta$ is the canonical odd subalgebra of $(\QSym, \zeta_Q)$, as in \cite[\S 6]{abs}.

Recall from Example~\ref{ex:even-odd-basis} that the basis $\{\mathfrak{e}_\a\}$ arises from the ordered partition construction with the partition $\mathbb{E} \prec \mathbb{O}$ and with $f_\mathbb{E}(\a) = f_\mathbb{O}(\a) = \f{1}{\ell(\a)!}.$ Here we consider a larger class of shuffle bases (equivalently, quasisymmetric power sums) where $f_\mathbb{E}$ can be any nonsingular shuffle character. In other words, we consider shuffle bases $\{X_\a\}$ of the form \begin{equation} \label{equation-Salpha}
    X_\a = \sum_{\b \ge \a} f(\a, \b) M_\b
\end{equation} where \begin{equation} \label{equation-Salpha-f}
    f(\a) = \begin{cases}
         f_\mathbb{E}(\a|_\mathbb{E}) \frac{1}{\odd(\a)!} & \text{if } \a = \a|_\mathbb{E} \, \a|_\mathbb{O}, \\
        0 & \text{otherwise.}
    \end{cases}
\end{equation}
We then have the following generalization of \cite[Theorem 3.8]{al}.

\begin{thm}
    The basis $\{X_\a\}$ defined by \eqref{equation-Salpha} and \eqref{equation-Salpha-f} is an eigenbasis for $\Theta,$ with \[
        \Theta(X_\a) = \begin{cases}
            2^{\ell(\a)} X_\a & \text{if } \a \text{ is odd}, \\
            0 & \text{otherwise.}
        \end{cases}
    \]
\end{thm}

\begin{proof}
    The main claim is that \begin{equation} \label{eq:nuQ_Salpha}
        \nu_Q(X_\a) = \begin{cases}
            2^{\ell(\a)} f(\a) & \text{if } \a \text{ is odd}, \\
            0 & \text{otherwise.}
        \end{cases}
    \end{equation} We first show how this implies the theorem. By Proposition~\ref{prop:abs-universal}, we have \[
        \Theta(X_\a)
        = \sum_{\b \comp n} \left(\nu_Q^{\tensor \ell(\b)} \Delta_\b(X_\a)\right) M_\b = \sum_{\b \ge \a} \nu_Q(X_{\a^{(1)}}) \dotsm \nu_Q(X_{\a^{(\ell(\b))}}) M_\b.
    \] If $\a$ is not odd, then for every $\b \ge \a,$ there exists some $\a^{(i)}$ which is not odd, and so every term of this sum vanishes. If instead $\a$ is odd, then so is every $\a^{(i)}$ and therefore \[
        \Theta(X_\a)
        = \sum_{\b \ge \a} \left(\prod_{i=1}^{\ell(\b)} 2^{\ell(\a^{(i)})} f(\a^{(i)}) \right) M_\b
        = 2^{\ell(\a)} \sum_{\b \ge \a} f(\a, \b) M_\b
        = 2^{\ell(\a)} X_\a,
    \] as desired.
    
    It remains to establish \eqref{eq:nuQ_Salpha}. For that, we use $\nu_Q = \overline{\zeta}_Q^{-1} * \zeta_Q$ to write \[
        \nu_Q(X_\a)
        = \sum_{\b\y = \a} \overline{\zeta}_Q^{-1}(X_\b) \, \zeta_Q(X_\y)
        = \sum_{\b\y = \a} \overline{\zeta}_Q \circ S_{\QSym}(X_\b) \, \zeta_Q(X_\y).
    \]
                    Because $\{X_\a\}$ is a shuffle basis, we can apply \eqref{eq:Sh-antipode} to get $S_\QSym(X_\b) = (-1)^{\ell(\b)} X_{\tiny \rev{\b}}.$ Therefore, we have $\overline{\zeta}_Q \circ S_\QSym(X_\b) = (-1)^{|\b|+\ell(\b)} \zeta_Q(X_{\tiny \rev{\b}}) = (-1)^{\even(\b)} f(\rev{\b}),$ so \begin{equation} \label{eq:eigenbasis-computation}
        \nu_Q(X_\a)
        = \sum_{\b\y = \a} (-1)^{\even(\b)} f(\rev{\b}) f(\y).
    \end{equation} Observe that $f(\y) = 0$ unless $\y = \y|_\mathbb{E} \y|_\mathbb{O},$ and $f(\rev{\b}) = 0$ unless $\b = \b|_\mathbb{O} \b|_\mathbb{E}.$ Since $\a = \b\y,$ it follows that $\nu_Q(X_\a) = 0$ unless $\a$ has the form $\a = \a^{(1)} \a^{(2)} \a^{(3)}$ for $\a^{(1)}, \a^{(3)}$ odd and $\a^{(2)}$ even.

    We now show that if $\a^{(2)} \neq \empty$, then $\nu_Q(X_\a) = 0.$ In this case, the nonzero terms in the above sum have $\b = \a^{(1)} \b'$ and $\y = \y' \a^{(3)},$ where $\b'\y' = \a^{(2)}.$ Therefore we have \[ \begin{aligned}
        \nu_Q(X_\a)
        = \sum_{\b'\y' = \a^{(2)}} (-1)^{\even(\a^{(1)} \b')} f(\rev{\b'}\rev{\a^{(1)}}) f(\y'\a^{(3)})
        &= \sum_{\b'\y' = \a^{(2)}} (-1)^{\even(\b')} f(\rev{\b'})\f{1}{\ell(\a^{(1)})!} f(\y') \f{1}{\ell(\a^{(3)})!} \\
        &= \f{1}{\ell(\a^{(1)})! \, \ell(\a^{(3)})!} \nu_Q(X_{\a^{(2)}}).
    \end{aligned} \] Since $\a^{(2)}$ is even, every composition $\delta \ge \a^{(2)}$ has $\lp(\delta)$ even, so it follows from \eqref{eq:nuQ_Malpha} that $\nu_Q(X_{\a^{(2)}}) = \sum_{\delta \ge \a^{(2)}} f(\a^{(2)},\delta) \nu_Q(M_\delta) = 0.$ Thus, $\nu_Q(X_\a) = 0$ if $\a^{(2)} \neq \empty,$ as desired.

    The only remaining case is when $\a^{(2)} = \empty,$ i.e., $\a$ is odd. In this case, \eqref{eq:eigenbasis-computation} becomes \[
        \nu_Q(X_\a)
        = \sum_{\b\y = \a} \f{1}{\ell(\b)! \ell(\y)!}
        = \sum_{i=0}^{\ell(\a)} \f{1}{i! (\ell(\a)-i)!}
        = \f{2^{\ell(\a)}}{\ell(\a)!}
        = 2^{\ell(\a)} f(\a),
    \] as claimed. This completes the proof.
\end{proof}

Next, we consider the quasisymmetric infinitesimal character $g$ that corresponds to a shuffle basis $X_\a$ given by \eqref{equation-Salpha} and \eqref{equation-Salpha-f}. When $\a$ is a composition of odd size, there is a simple formula for $g(\a)$ that does not depend on $f_\mathbb{E},$ as we now establish.

\begin{thm}
    Let $g \colon \Compne \to \bk$ be the map such that \begin{equation} \label{eq:equation-Salpha-g}
        M_\a = \sum_{\b \ge \a} g(\a, \b) X_\b,
    \end{equation} where $\{X_\a\}$ is given by \eqref{equation-Salpha} and \eqref{equation-Salpha-f}. For $|\a|$ odd, \begin{equation} \label{eq:evenodd-gfunction-claim}
        g(\a) = \begin{cases}
            \displaystyle\f{(-1)^{\ell(\a)-1}}{\odd(\a)} & \text{if } \lp(\a) \text{ is odd}, \\
            0 & \text{otherwise.}
        \end{cases}
    \end{equation}
\end{thm}

\begin{proof}
    Suppose that $\a$ is a composition with $n = |\a|$ odd. Applying $\nu_Q$ to both sides of \eqref{eq:equation-Salpha-g}, we have by \eqref{eq:nuQ_Malpha} and \eqref{eq:nuQ_Salpha} that \begin{equation} \label{eq:Salpha-g-proof}
        \sum_{\text{odd } \b \ge \a} g(\a,\b) 2^{\ell(\b)} \f{1}{\ell(\b)!}
        = \begin{cases}
            2 (-1)^{|\a| + \ell(\a)} & \text{if } \lp(\a) \text{ is odd}, \\
            0 & \text{otherwise.}
        \end{cases}
    \end{equation} Since $n$ is odd, the left-hand side has a term $2g(\a),$ while the other terms depend only on the values $g(\y)$ where $|\y|$ is odd and less than $n.$ Therefore, \eqref{eq:Salpha-g-proof} uniquely determines $g(\a)$ for all $\a$ of odd size. It follows that to show that \eqref{eq:evenodd-gfunction-claim} holds, we need only show that it makes \eqref{eq:Salpha-g-proof} true. In other words, it suffices to prove that \[
        \sum_{
            \substack{
                \text{odd } \b \ge \a \\
                \text{all } \lp(\a^{(i)}) \text{ odd}
            }
        } \f{(-1)^{\ell(\a)-\ell(\b)}}{\odd(\a,\b)} 2^{\ell(\b)} \f{1}{\ell(\b)!}
        = \begin{cases}
            2 (-1)^{|\a| + \ell(\a)} & \text{if } \lp(\a) \text{ is odd}, \\
            0 & \text{otherwise}
        \end{cases}
    \] whenever $|\a|$ is odd. The left-hand side is clearly zero if $\lp(\a)$ is even, so assume that $\lp(\a)$ is odd. Then we want to show \[
        \sum_{
            \substack{
                \text{odd } \b \ge \a \\
                \text{all } \lp(\a^{(i)}) \text{ odd}
            }
        } \f{(-2)^{\ell(\b)}}{\odd(\a,\b)} \f{1}{\ell(\b)!}
        = 2 (-1)^{|\a|}.
    \] 
    In order for $\lp(\a^{(i)})$ to be odd for all $i$, each even part of $\a$ must lie in the same $\a^{(i)}$ as the closest odd part to its right. In addition, the summand on the left-hand side only depends on $\ell(\b)$ and the numbers $\odd(\a^{(i)}),$ so it follows that the left-hand side only depends on $\odd(\a).$ Therefore, we may assume without loss of generality that $\a = 1^m,$ where $m = \odd(\a).$ Then the left-hand side becomes \[
        \sum_{\text{odd } \b \, \comp \, m} \f{(-2)^{\ell(\b)}}{\p(\b)} \f{1}{\ell(\b)!}.
    \] This sum is the coefficient of $x^m$ in the power series \[
        \exp \left( \sum_{k \text{ odd}} (-2) \f{x^k}{k} \right) = \exp(\log(1-x) - \log(1+x)) = \frac{1-x}{1+x},
    \] which is $2(-1)^m = 2(-1)^{|\a|},$ as desired.
\end{proof}

\subsection{The combinatorial basis}

Recall that the combinatorial quasisymmetric power sums $\{\mathfrak{p}_\a\}$ (as well as the reverse combinatorial basis $\{\mathfrak{p}_\a^r\}$ of Remark~\ref{rmk:reverse-combinatorial-qps}) have the special property that they expand into the monomial basis with nonnegative integer coefficients. We now determine all such QPS bases.

\begin{thm} \label{thm:integral-qps}
    Let $\{P_\a\}$ be a graded basis of QSym. The following are equivalent:
    \begin{enumerate}[(i)]
        \item $\{P_\a\}$ is a QPS basis, and $P_\a = \sum_\b c_{\a\b} M_\b$ for some integers $c_{\a\b} \ge 0.$
        \item There exists a total order $\prec$ on $\Zpos$ such that \[
            P_\a = \aut(\a) \sum_{\a \le' \b} \frac{1}{\aut(\a,\b)} M_\b,
        \] where $\a \le' \b$ if $\a \le \b$ and every $\a^{(i)}$ has weakly increasing parts with respect to $\prec.$
    \end{enumerate}
\end{thm}

\begin{proof}
    By the ordered partition construction (Theorem~\ref{thm:ordered-partition}), the formula in (ii) defines a quasisymmetric power sum basis $\{P_\a\},$ and \eqref{eq:aut-quotient} implies that the coefficients of $P_\a$ in the monomial basis are nonnegative integers. So, it remains to show that (i) implies (ii).
    
    Suppose that (i) holds. Let $f$ be the normalized shuffle character of $\{P_\a\},$ so that \begin{equation} \label{eq:nonneg-f}
        P_\a = \aut(\a) \sum_{\b \ge \a} f(\a, \b) M_\b.
    \end{equation} Then $\aut(\a) f(\a, \b)$ must be a nonnegative integer for all $\a \le \b$. Taking $\b = |\a|,$ we see that $\aut(\a) f(\a)$ must be a nonnegative integer for all $\a.$ This condition is also sufficient: if $\a \le \b,$ then \[
        \aut(\a) f(\a, \b) = \frac{\aut(\a)}{\aut(\a,\b)} \prod_{i=1}^{\ell(\a)} \aut(\a^{(i)}) f(\a^{(i)})
    \] which is a nonnegative integer again by \eqref{eq:aut-quotient}. Therefore, we must determine all shuffle characters $f$ such that $\aut(\a) f(\a)$ is always a nonnegative integer.
    
    Let $p$ and $q$ be distinct positive integers. Then, $f(pq) + f(qp) = f(p)f(q) = 1,$ where $pq$ and $qp$ denote compositions with two parts. Because $\aut(pq) = \aut(qp) = 1,$ both $f(pq)$ and $f(qp)$ must be nonnegative integers, so they must be $0$ and $1$ in some order. Define a relation $\prec$ on $\Zpos$ by declaring that $p \prec q$ if and only if $f(pq) = 1$ (and $f(qp) = 0$). We claim that $\prec$ is a total order (in particular, it is transitive) and that \[
        f(\a) = \begin{cases}
            \frac{1}{\aut(\a)} & \text{if } \a_1 \preceq \dotsm \preceq \a_{\ell(\a)}, \\
            0 & \text{otherwise}.
        \end{cases}
    \] This, combined with \eqref{eq:nonneg-f}, will establish (ii).

    We make the following observation. Suppose that $p \prec q,$ so $f(qp) = 0.$ For any composition $\b,$ we have \[
        0 = f(\b) f(qp) = \sum_{\a \in \b \sh qp} f(\a),
    \] which forces $f(\a) = 0$ for all $\a \in \b \sh qp$ (since the values of $f$ are nonnegative). It follows that $f(\a) = 0$ if $\a$ has some subsequence $qp$ where $p \prec q.$ From this, we can deduce that $\prec$ is transitive: if $a \prec b$ and $b \prec c$ but $c \prec a$, then \[
        1 = f(a) f(bc) = f(abc) + f(bac) + f(bca) = 0 + 0 + 0 = 0,
    \] a contradiction. Thus, $f(\a) = 0$ unless $\a_1 \preceq \dotsm \preceq \a_{\ell(\a)}.$ 

    Finally, suppose that the parts of $\a$ are weakly increasing with respect to $\prec$. Then, writing \[
        1 = f(\a_1) \dotsm f(\a_{\ell(\a)}) = \sum_{\y \in \a_1 \sh \dotsm \sh \a_{\ell(\a)}} f(\y),
    \] we see that the terms on the right-hand side are all zero except for $\aut(\a)$ occurrences of $f(\a).$ Therefore $f(\a) = \f{1}{\aut(\a)},$ as desired.
\end{proof}

\printbibliography

\end{document}